\documentclass[12pt]{article}
\usepackage{amssymb,amsmath}
\usepackage[affil-it]{authblk}
\providecommand{\keywords}[1]{\textbf{\textit{Keywords:}} #1}
\providecommand{\udk}[1]{\begin{flushleft}\textbf{\textit{UDK }} #1\end{flushleft}}
\usepackage{mathrsfs, amssymb, array,fancybox, upgreek}
\usepackage[all, matrix, arrow, curve]{xy}

\usepackage{tabularx,multirow}
\usepackage{setspace}
\usepackage{amsthm}
\usepackage{hyperref}
\usepackage[yyyymmdd,hhmmss]{datetime}
\usepackage[matrix,arrow,curve]{xy}
\usepackage{adjustbox}
\makeatletter
\@addtoreset{equation}{subsection}
\makeatother

\newcommand{\CC}{\mathbb{C}}
\newcommand{\ZZ}{\mathbb{Z}}
\newcommand{\PP}{\mathbb{P}}

\newcommand{\FF}{\mathbb{F}}

\newcommand{\Sym}{\mathfrak{S}}

\newcommand{\OOO}{{\mathscr{O}}}

\newcommand{\rk}{\operatorname{rk}}

\newcommand{\Bir}{\operatorname{Bir}}
\newcommand{\Pic}{\operatorname{Pic}}

\newcommand{\Aut}{\operatorname{Aut}}

\newcommand{\Cr}{\operatorname{Cr}}

\newcommand{\PGL}{{\operatorname{PGL}}}
\newcommand{\PSL}{{\operatorname{PSL}}}
\newcommand{\SL}{{\operatorname{SL}}}
\newcommand{\Sp}{{\operatorname{Sp}}}
\newcommand{\Q}{{\operatorname{Q}}}
\newcommand{\PSp}{{\operatorname{PSp}}}
\newcommand{\Tr}{{\operatorname{Tr}}}
\newcommand{\GL}{{\operatorname{GL}}}
\newcommand{\z}{{\operatorname{z}}}

\newcommand{\SM}{{\operatorname{M}}}
\newcommand{\Ort}{{\operatorname{O}}}

\newcommand{\ed}{{\operatorname{ed}}}
\newcommand{\diag}{{\operatorname{diag}}}

\newcommand{\Alt}{{\mathfrak{A}}}
\newcommand{\mumu}{{\boldsymbol{\mu}}}

\newcommand{\xref}[1]{{\rm\ref{#1}}}

\swapnumbers
\theoremstyle{plain}
\newtheorem{theorem}[subsection]{Theorem}
\newtheorem{lemma}[subsection]{Lemma}
\newtheorem{proposition}[subsection]{Proposition}

\newtheorem{corollary}[subsection]{Corollary}
\newtheorem{scorollary}[equation]{Corollary}
\newtheorem*{claim*}{Claim}

\newtheorem{slemma}[equation]{Lemma}

\theoremstyle{definition}

\newtheorem{definition}[subsection]{Definition}
\newtheorem*{definition*}{Definition}

\newtheorem{example-remark}[subsection]{Remark-Example}
\newtheorem{subexample-remark}[equation]{Remark-Example}

\newenvironment{case*}{\par\smallskip\noindent}{}

\newtheorem{notation}[subsection]{Notation}
\newtheorem*{notation*}{Notation}

\newtheorem{remark}[subsection]{Remark}

\newtheorem{sremark}[equation]{Remark}

\newtheorem{construction}[subsection]{Construction}

\newcounter{NN}

\newcounter{NO}

\title{Quasi-simple finite groups of essential dimension $3$}
\author{Yuri Prokhorov
\thanks{The author was were partially supported by RFBR 15-01-02164 and 15-01-02158
and by the Russian Academic Excellence Project ``5-100''.
}
}

\affil{\small
Steklov Mathematical Institute, Russia
\\
Moscow State Lomonosov 
University, Russia
\\
National Research University Higher School of Economics, Russia
\\
\tt{e-mail: prokhoro@mi.ras.ru}
}

\date{}

\begin{document}\maketitle

\begin{center}
 \it  To the memory of Professor 	
Alfred Lvovich Shmel'kin 
\end{center}
\udk{512.76}
\keywords{essential dimension, group, algebraic variety, representation, Cremona group
}

 \begin{abstract}
We classify quasi-simple finite groups of essential dimension 3.
 \end{abstract}

\section{Introduction}
This paper is based on the author's talk given at the Magadan conference.

Let $G$ be a finite group and let $V$ be a faithful representation of $G$ regarded as 
an algebraic variety.
A \textit{compression} is a $G$-equivariant dominant rational map $V \dashrightarrow X$ 
of faithful $G$-varieties. The
\textit{essential dimension} of $G$, denoted $\ed(G)$, is the minimal dimension of all 
faithful
$G$-varieties $X$ appearing in compressions $V \dashrightarrow X$.
This notion was introduced by J. Buhler and Z. Reichstein \cite{Buhler-Reichstein-1997}
in relation to some classical problems in the theory of polynomials.
It turns out that the essential dimension
depends only on the group $G$, i.e. it does not depend on 
the choice of linear representation $V$ \cite[Theorem 3.1]{Buhler-Reichstein-1997}.

The computation of the essential dimension is a challenging 
problem of algebra and algebraic geometry. 
The finite groups of essential dimension $\le 2$ have been classified (see
\cite{Duncan2013}). Simple finite groups 
of essential dimension $3$ have been recently determined by A. Beauville 
\cite{Beauville2014} (see also \cite{Serre-2008-2009}, \cite{Duncan2010}):

\begin{theorem}
The simple groups of essential dimension $3$ are $\Alt_6$ and possibly
$\PSL_2(11)$.
\end{theorem}

The essential dimension of $p$-groups was computed by Karpenko and Merkurjev
\cite{Karpenko-Merkurjev-2008}.
In this short note we find all the finite quasi-simple
groups of essential dimension $3$.

\begin{definition}
A group $G$ is said to be \emph{quasi-simple} if 
$G$ is perfect, that is, it equals its  commutator subgroup, and the quotient 
of $G$ by its center  is a simple non-abelian group.
\end{definition}

The main result of this paper is the following.

\begin{theorem}\label{main}
Let $G$ be a finite quasi-simple non-simple group.
If $\ed(G)=2$, then $G\simeq 2.\Alt_5$.
If $\ed(G)=3$, then $G\simeq 3.\Alt_6$.
\end{theorem}

\begin{notation} 
Throughout this paper the ground field is supposed to be the field of complex numbers $\CC$. 
We employ the following standard notations used in the group theory.

\begin{itemize}
\item 
$\mumu_n$ denotes the multiplicative group of order $n$ (in $\CC^*$),
\item 
$\Alt_n$ denotes the alternating group of degree $n$,
\item 
$\SL_n (q)$ (resp. $\PSL_n (q)$) denotes the special linear group (resp. projective special linear group) over
the finite field $\mathbf F_q$,
\item 
$n.G$ denotes a non-split central extension of $G$ by $\mumu_n$, 
\item 
$\z(G)$ (resp. $[G, G]$) denotes the center (resp. 
the commutator subgroup) of a group $G$.
\end{itemize}
All simple groups are supposed to be non-cyclic.
\end{notation}

\section{Proof of Theorem \ref{main}}
The following assertion is an immediate consequence of the 
corresponding fact for simple groups \cite{Prokhorov2009e}.
\begin{proposition}
\label{Proposition-list-intro}
Let $X$ be a three-dimensional rationally connected 
variety and let $G\subset \Bir(X)$ be a finite quasi-simple non-simple group.
Then $G$ is isomorphic to one of the following:
\begin{equation}
\label{equation-list}
\SL_2(7),\ \SL_2(11),\ \Sp_4(3),\ 
2.\Alt_5,\ 
n.\Alt_6,\ n.\Alt_7\ \text{with $n=2,3,6$.}
\end{equation}
\end{proposition}

\begin{proof}
We may assume that $G$ biregularly (and faithfully) acts on $X$.
Let $Y:= X/\z(G)$ and $G_1:= G/\z(G)$. Then $Y$ is a three-dimensional rationally connected variety 
acted on by a finite simple group $G_1$. Then according to \cite{Prokhorov2009e}
$G_1$ belongs to the following list:
\begin{equation}
\label{equation-simple-list}
\PSL_2(7),\ \PSL_2(8),\ \PSL_2(11),\ \PSp_4(3),\ 
\Alt_5,\ 
\Alt_6,\ \Alt_7.
\end{equation}
Since $G_1$ is perfect, there exists the\textit{ universal covering group}
$\tilde G_1$, that is, a central extension of $G_1$ such that
for any other extension $\hat G_1$
there is a unique homomorphism $\tilde G_1\to \hat G_1$ of central extensions 
(see e.g. \cite[\S 11.7, Theorem 7.4]{Karpilovsky-rep-2}). The kernel of $\tilde G_1\to G_1$
is the Schur multiplier $\SM(G_1)=H^2(G_1,\CC^*)$ of $G_1$.
Thus $G$ is uniquely (up to isomorphism)
determined by $G/\z(G)$ and the homomorphism $ \SM(G_1)\to \z(G)$.
It is known that $\SM(G_1)\simeq \mumu_2$ in all the cases \eqref{equation-simple-list}
except for $\Alt_6$ and $\Alt_7$ where the Schur multiplier is isomorphic to
$\mumu_6$, and $\PSL_2(8)$ where the Schur multiplier 
is trivial (see \cite[\S 12.3, Theorem 3.2, \S 16.3, Theorem 3.2]{Karpilovsky-rep-2}
and \cite{atlas}). This gives us the list \eqref{equation-list}.
\end{proof}

\begin{sremark}
We do not assert that all the possibilities  \eqref{equation-list} occur.
Using the technique developed in the works \cite{Prokhorov2009e},
\cite{Prokhorov-Shramov-p-groups}, \cite{Prokhorov-Shramov-J-const}, \cite{Prokhorov-Shramov-3folds} 
it should be possible to obtain a complete classification 
of actions of quasi-simple groups on rationally connected threefolds.    
However, this is much more difficult problem.
\end{sremark}

\begin{scorollary}\label{corollary-representations}
Let $G$ be a group satisfying the conditions of \xref{Proposition-list-intro}.
Then $G$ has an irreducible faithful representation.
The minimal dimension of such a representation $V$ is given by the 
following table \textup{\cite{atlas}}:
\setlength{\extrarowheight}{5pt}
\par\medskip\noindent
\scalebox{0.95}{
\begin{tabular}{c|c|c|c|c}
$G$& $2.\Alt_5$& $3.\Alt_6$ & 
$\SL_2(7),\, \Sp_4(3),\, 2.\Alt_6,\,2.\Alt_7$ 
&
$\SL_2(11),\, 6.\Alt_6,\, 3.\Alt_7,\, 6.\Alt_7$
\\\hline
$\dim V$ & $2$ &$3$& $4$ &$6$
\end{tabular}
}
\end{scorollary}

\begin{construction}\label{construction}
Let $G$ be a finite quasi-simple non-simple group having a faithful irreducible representation $V$.
Let $\psi: V \dashrightarrow X$
be a compression with $\dim X=\ed(G)$. Applying an equivariant resolution of singularities 
(see \cite{Abramovich-Wang}) we may assume that $X$ is smooth (and projective). 
Furthermore, consider the compactification $\PP:=\PP(V\oplus \CC)$ and let 
$X\overset{\varphi}\longleftarrow Y\overset{f}\longrightarrow \PP\supset V$
be an equivariant resolution of $\psi$, where $f$ is a birational morphism and $Y$ is smooth and projective. 
We also may assume that $f$ is passed through 
the blowup $\tilde f: \tilde \PP\to \PP$ of $0\in V\subset \PP$. Let $\tilde E\subset \tilde V$ 
be the $\tilde f$-exceptional divisor
and let $E\subset Y$ be its proper transform, and let $B:=\varphi(E)$.
Thus we have the following $G$-equivariant diagram:
\vspace{10pt}
\[
\vcenter{
\xymatrix{
&E\ar[dl]\ar@{}[r]|-*[@]{\subset}\ar@/^16pt/[rrr]&Y\ar[dr]^{f}\ar[dl]_{\varphi}\ar[r]& 
\tilde \PP\ar[d]^{\tilde f}\ar@{}[r]|-*[@]{\supset}&\tilde E\ar[dr]
\\
B\ar@{}[r]|-*[@]{\subset}&X&&\PP\ar@{-->}[ll]_{\psi}\ar@{}[r]|-*[@]{\supset}&V\ar@{}[r]|-*[@]{\ni}&0
} 
}
\]
The action of $\z(G)$ on $\tilde E\simeq \PP(V)$ and on $E$ is trivial
because $V$ is an irreducible representation.
Hence, $G/\z(G)$ faithfully acts on $E$.
By assumption the action of $\z(G)$ on $X$ is
faithful.  
Hence  $B\neq X$ and so 
$B$ is a rationally connected variety of dimension $<\ed(G)$. 
\end{construction}

\begin{proposition}\label{proposition-center}
Let $G$ be a finite quasi-simple non-simple group
having a faithful irreducible representation. 
Then $G/\z(G)$ acts faithfully of a rationally connected variety 
of dimension $<\ed(G)$.
\end{proposition}

\begin{proof}
Let $V$ be a faithful irreducible representation of $G$ and 
let $\psi : V\dashrightarrow X$ be a compression with $\dim X=\ed(G)$.
Apply the construction \ref{construction}.
Assume that $G$ has a fixed point $P\in X$. Then $G$ has faithful representation 
on the tangent space $T_{P,X}$. Let $T_{P,X}=\oplus T_i$ 
be the decomposition in irreducible components. At least one of them,
say $T_1$ is non-trivial. Then $G/\z(G)$ faithfully acts on $\PP(T_1)$,
where $\PP(T_1)< \dim X=\ed(G)$. 
Thus we may assume that $G$ has no fixed points on $X$.
By the construction \ref{construction} the variety $B$ is  rationally connected
and $\dim B<\ed(G)$. 
Since $G$ has no fixed points on $B$ and the group $G/\z(G)$ is simple, its action on $B$ must be effective.
\end{proof}

Comparing the list \eqref{equation-list} with Theorem \ref{theorem-2-dimensional-simple}
we obtain the following.

\begin{corollary}
Let $G$ be a finite quasi-simple non-simple group with
$\ed(G)\le 3$. 
Then for $G$ we have one of the following possibilities:
\begin{equation}
\label{equation-list-2}
\SL_2(7),\qquad n.\Alt_6\ \text{with $n=2,3,6$.}
\end{equation}
\end{corollary}

Now we consider the possibilities of \eqref{equation-list-2}
case by case.
\begin{lemma}
$\ed(2.\Alt_5)=2$ and $\ed(3.\Alt_6)=3$.
\end{lemma}
\begin{proof}
Let us prove, for example, the second equality.
Since $3.\Alt_6$ has a faithful there-dimensional representation,
$\ed(3.\Alt_6)\le 3$. On the other hand, $\Alt_6$ cannot effectively act on 
a rational curve. Hence, by Proposition \ref {proposition-center} \ $\ed(3.\Alt_6)\ge 3$.
\end{proof}

\begin{slemma}\label{lemma-G}
Let $G$ be a quasi-simple non-simple group.
Assume that $G\not\simeq 2.\Alt_5,\, 3.\Alt_6$. Assume also that 
$G$ contains a subgroup $\bar H$ 
such that
\begin{enumerate}
\item \label{lemma-G-1}
$\bar H$ is not abelian but its image $H\subset G/\z(G)$ is abelian,
\item \label{lemma-G-2}
for any action of $G/\z(G)$ on a rational projective surface 
the subgroup $H\subset G/\z(G)$ has a fixed point. 
\end{enumerate}
Then $\ed (G)\ge 4$.
\end{slemma}
\begin{proof}
Since $\bar H/(\z(G)\cap \bar H)=H$, 
we have $\z(G)\cap \bar H\supset [\bar H,\bar H]$ and $[\bar H,\bar H]\neq \{1\}$
(because $\bar H$ is not abelian).
Assume that $\ed (G)=3$.
Apply construction \ref{construction}.
From the list \eqref{equation-list} one can see that $G/\z(G)$ 
cannot faithfully act on a rational curve and 
by Corollary \ref{corollary-representations}\ $G$ has no fixed points on $X$.
Hence, $B$ is a (rational) surface.
By \ref{lemma-G-2}
the group $\bar H$ has a fixed point, say $P$, on $B\subset X$. There is an invariant 
decomposition $T_{P, X}=T_{P,B}\oplus T_1$, where $\dim T_1=1$.
The action of $[\bar H,\bar H]$ on $T_{P,B}$ and $T_1$ is trivial.
Hence it is trivial on $T_{P, X}$ and $X$, a contradiction. 
\end{proof}

\begin{proposition}\label{proposition-SL27}
$\ed(\SL_2(7))= 4$.
\end{proposition}

\begin{slemma}\label{lemma-SL27-fixed point}
Let $S$ be a smooth projective rational surface admitting the action 
of $\PSL_2(7)$. Let $H\subset \PSL_2(7)$ be a subgroup isomorphic to
$\mumu_2\times \mumu_2$. Then $H$ has a fixed point on $S$.
\end{slemma}

\begin{proof}
Since $H$ is abelian, according to \cite{Kollar-Szabo-2000} it is sufficient to show 
the existence of a fixed point on some birational model of $S$.
By Theorem \ref{theorem-2-dimensional-simple} we may assume that 
$S$ is either $\PP^2$ or some special del Pezzo surface of degree $2$
(see \ref{theorem-2-dimensional-simple}\ref{theorem-2-dimensional-simple-2}).
In the former case, $\PP^2=\PP(W)$, where $W$ is a three-dimensional irreducible 
representation of $\PSL_2(7)$. Then  the abelian group $H\simeq \mumu_2\times \mumu_2$ has 
a fixed point on $\PP^2=\PP(W)$. Thus we assume that $S$ is a del Pezzo surface of degree $2$.

Let $\upalpha\in H$ be an  element of order $2$.
First assume that 
$\upalpha$ has a curve $C$ of fixed points. 
The image $\pi(C)$ under the anticanonical double cover
$\pi: S\to \PP^2$ must be a line (because the action on $\PP^2$ is linear). 
Let $\upalpha'\in H$, $\upalpha'\neq \upalpha$ be another element  of order $2$.
Then $\upalpha'(C)$ is also a curve $C$ of $\upalpha$-fixed points
and $\pi(\upalpha'(C))$ is also a line. 
Hence, $\pi(\upalpha'(C))=\pi(C)$ and  $\pi^{-1}(\pi(C))$
contains $\upalpha'(C)$ and $C$. Since $\pi^{-1}(\pi(C))\sim -K_S$  and 
the fixed point locus of $\upalpha$ is smooth, we have $\upalpha'(C)=C=\pi^{-1}(\pi(C))\sim -K_S$
and it is an ample divisor. Note that all the elements of order $2$ are conjugate in $\PSL_2(7)$.
Hence $\upalpha'$ also has a curve  of fixed points, say $C'$, and $C'\sim -K_S$.
Then the intersection points $C\cap C'$ are fixed by $H=\langle\upalpha,\, \upalpha'\rangle$.

Thus we  may assume that  any element $\upalpha\in H$ of order $2$
has only isolated fixed points.
The holomorphic Lefschetz 
fixed point formula shows that the number of these fixed points equals $4\chi(\OOO_S)=4$.
Then by the topological Lefschetz 
fixed point formula 
\[
\Tr_{H^2(S,\CC)} \upalpha^*=2.
\]
Since $\dim H^2(S,\CC)=8$ and all the eigenvalues of $\upalpha^*$ equal $\pm 1$, 
its determinant  must be equal to $-1$ and so we have a non-trivial character 
of the group $\PSL_2(7)$.  
This contradicts the the fact that $\PSL_2(7)$ is simple.
\end{proof}

\begin{proof}[Proof of Proposition \xref{proposition-SL27}]
Consider the subgroup $\bar H\subset \SL_2(7)$ generated by the matrices 
\[
A=
\begin{pmatrix}
1 &1 \cr
5 &-1 \cr
\end{pmatrix}
\qquad
B=\begin{pmatrix}
1 &5 \cr
1 &-1 \cr
\end{pmatrix}
\]
Let $H$ be its image in $\PSL_2(7)$.
It is easy to check that $A^2=B^2=-I$ and $[A,B]=-I$. Hence $\bar H$ is isomorphic to the quaternion group $\Q_8$
and $H\simeq \Q_8/\z(\Q_8)\simeq \mumu_2\times \mumu_2$. By Lemma \ref{lemma-SL27-fixed point}
the group $\bar H$ has a fixed point on $B\subset X$. 
Hence we can apply Lemma \ref{lemma-G}.
\end{proof}

\begin{proposition}\label{proposition-A6}
$\ed(2.\Alt_6)= 4$, $\ed(6.\Alt_6)\ge 4$.
\end{proposition}
\begin{proof}
We are going to apply Lemma \ref{lemma-G}.
Let $S$ be a projective rational surface acted by $G/\z(G)=\Alt_6$.
By Theorem \ref{theorem-2-dimensional-simple} we may assume that $S\simeq \PP^2$.
The group $2.\Alt_6$ is isomorphic to $\SL_2(9)$ (see \cite{atlas}).
As in the proof of Proposition \xref{proposition-SL27} take the subgroup
$\bar H\subset \SL_2(9)$ generated by the matrices 
\begin{equation}
\label{equation-AB}
A=\begin{pmatrix}1&2\\ -1&-1\end{pmatrix}
\qquad
B=\begin{pmatrix}4&1\\ 1&-4\end{pmatrix}
\end{equation}
and  apply Lemma \ref{lemma-G}. 

In the case $G=6.\Alt_6$, let $Z\subset \z(G)$ be the subgroup of order $3$.
Then $6.\Alt_6/Z\simeq 2.\Alt_6\simeq \SL_2(9)$.
Let $\hat H$ be the inverse image of the subgroup of
$\SL_2(9)$ generated by $A$ and $B$ from \eqref{equation-AB} and let $\bar H\subset \hat H$ be the Sylow $2$-subgroup.
Then, as above, $\bar H\simeq \Q_8$ and we can apply Lemma \ref{lemma-G}.
\end{proof}

\begin{sremark}
Since $6.\Alt_6$ has a six-dimensional faithful representation, one has 
$\ed(6.\Alt_6)\le 6$. However, we are not able to compute the precise value.
\end{sremark}

\section{Appendix: simple subgroups in the plane Cremona group}

The following theorem can be easily extracted from the classification \cite{Dolgachev-Iskovskikh}.
For convenience of the reader we provide a relatively short and
self-contained proof. 

\begin{theorem}[{\cite{Dolgachev-Iskovskikh}}]
\label{theorem-2-dimensional-simple}
Let $G\subset \Cr_2(\CC)$ be a finite simple subgroup.
Then the embedding $G\subset \Cr_2(\CC)$ is induced by 
one of the following actions:
\begin{enumerate}
\item \label{theorem-2-dimensional-simple-P2}
$\Alt_5$, $\PSL_2(7)$, or $\Alt_6$ acting on $\PP^2$,
\item \label{theorem-2-dimensional-simple-5}
$\Alt_5$ acting on the del Pezzo surface of degree $5$;
\item \label{theorem-2-dimensional-simple-2}
$\PSL_2(7)$ acting on some special del Pezzo surface of degree $2$
which can be realized as a double cover of $\PP^2$ branched in the 
Klein quartic curve;

\item \label{theorem-2-dimensional-simple-P1P1}
$\Alt_5$ acting on $\PP^1\times \PP^1$ through the first factor.
\end{enumerate}
\end{theorem}

\begin{remark}
It is known (see \cite[\S 8]{Dolgachev-Iskovskikh} and \cite[\S B]{Cheltsov2014b}) that
the above listed actions are not conjugate in $\Cr_2(\CC)$.
\end{remark}

\begin{proof}
Applying the standard arguments (see e.g. \cite[\S 3]{Dolgachev-Iskovskikh}) we may assume that $G$ 
faithfully acts 
on a smooth projective rational surface $X$ which is either
a del Pezzo surface or an equivariant conic bundle.
Moreover, one has $\rk \Pic(X)^G=1$ (resp. $\rk \Pic(X)^G=1$) in the del Pezzo (resp. conic bundle)
case.

First, consider the case where 
$X$ has an equivariant conic bundle structure $\pi: X\to B\simeq \PP^1$.
Since the group $G$ is simple, it acts faithfully either on the base $B$ or on the
general fiber. Hence $G$ is embeddable to $\PGL_2(\CC)$.
This is possible only if $G\simeq \Alt_5$.
We claim that $\pi$ is a $\PP^1$-bundle. 
Assume that $\pi$ has a degenerate fiber $F$. Then its components 
$F', F''\subset F$ must be switched by an element $\upalpha\in G$ of order $2$.
The intersection point $P=F'\cap F''$ is fixed by $\upalpha$ and 
the action of $\upalpha$ on the tangent space $T_{P,X}$ is diagonalizable. 
Since the tangent directions to $F'$ and $F''$ are interchanged, 
the action of $\upalpha$ on $T_{P,X}$ has the form $\diag (1,-1)$.
This means that $\upalpha$ has a curve $C$ of fixed points passing through $P$.
Clearly, $C$ dominates $B$ and so the action of $\upalpha$ on $B$ is trivial.
All elements of order $2$ in $G\simeq \Alt_5$ are conjugate and generate $G$.
Hence $G$ acts on $B$ trivially. But then the point $P$ must be fixed by $G$ and 
$G$ has a faithful two-dimensional representation on $T_{P,X}$, a contradiction.

Thus, $\pi$ is a $\PP^1$-bundle and $X$ is a Hirzebruch surface $\FF_n$.
Let $F_1$ be a fiber and let $F_1, \dots, F_r$ be its orbit.
By making elementary transformations in $F_1, \dots, F_r$
one obtains an equivariant birational maps $\FF_n\dashrightarrow \FF_{n+r}$
and (only if $n\ge r$) $\FF_n\dashrightarrow \FF_{n-r}$.
There are such orbits for $r=12$, $20$, $30$, and $60$.
Using this trick we may replace $\FF_n$ with $\FF_{n'}$, where $n'=0$ or $1$.
If $n'=1$, then contracting the negative section we obtain
the action on $\PP^2$ with a fixed point. This is impossible for
$G\simeq \Alt_5$. Hence, we may assume that $X\simeq \PP^1\times \PP^1$.
Additional elementary transformations allow to trivialize the action on 
the second factor. This is the case \ref{theorem-2-dimensional-simple-P1P1}.
See \cite[\S B]{Cheltsov2014b} for details.

From now on we assume that $X$ is a del Pezzo surface\footnote{More generally, actions of simple groups on del Pezzo surfaces with log terminal singularities were 
studied in  \cite{Belousov2015}.} with 
$\rk\Pic(X)^G=1$. We consider the possibilities 
according to the degree $d =K_X^2$. 

\begin{case*}{\bf Case $d=1$.}
This case cannot occur, as $|-K_X|$ has one base point $P$, and 
$G$ has to act on $T_{P,X}$
effectively. Hence $G\subset \GL(T_{P,X})$.
However, there are no 
simple finite subgroups in $\GL_2(\CC)$, a contradiction. 
\end{case*}

\begin{case*}{\bf Case $d=2$.}
Then the anticanonical map 
$X\to \PP^2$ is a double cover whose 
branch divisor $B\subset \PP^2$ is a smooth quartic.
The action of $G$ in $X$ descends to $\PP^2$ so that 
$B$ is $G$-stable. Therefore, $G\subset \Aut(B)$.
According to the Hurwitz bound $| G| \le 168$.
Moreover, $\Aut(B)$ contains no elements of order $5$.
Then the only possibility is
$G\simeq \PSL_2( 7)$ and $B=\{x_1^3x_2+x_2^3x_3+x_3^3x_1=0\}$. 
We get the case \ref{theorem-2-dimensional-simple-2}.
\end{case*}

\begin{case*}{\bf Case $d=3$.}
Then $X$ is a cubic surface in $\PP^3$.
The action of $G$ on the lattice $\Lambda:=K_X^\perp\subset \Pic(X)$
is faithful. Hence our group $G$ has a representation on the vector space 
$\Lambda/2\Lambda=(\mathbf F_2)^6$ over the field $\mathbf F_2$. 
The intersection form induces an even quadratic form on $\Lambda$ 
and, therefore, it induces a quadratic form
\[
q(x):= \frac 12 (x,x) \mod 2.
\]
Take a standard basis $\mathbf h, \mathbf e_1,\dots, \mathbf e_6$ in $\Pic(X)$
with $(\mathbf h,\mathbf h)=1$, $(\mathbf h,\mathbf e_i)=0$, $(\mathbf e_i,\mathbf e_j)=-\delta_j^i$. 
Then using the basis 
$\mathbf e_1-\mathbf e_2$,
$\mathbf e_2-\mathbf e_3$,
$\mathbf e_4-\mathbf e_5$,
$\mathbf e_5-\mathbf e_6$,
$\mathbf h-\mathbf e_1-\mathbf e_2-\mathbf e_3$,
$\mathbf h-\mathbf e_4-\mathbf e_5-\mathbf e_6$ of $\Lambda$
we can write $q(x)$ in the following form:
\[
q(x)= x_1^2+ x_2^2+x_1x_2+ x_3^2+ x_4^2+x_3x_4+ x_5^2+ x_6^2+x_5x_6. 
\]
Then it is easy to see that the Arf invariant of $q(x)$ equals $1$.
The group preserves the intersection form and the quadratic form $q(x)$.
Therefore, there is a natural embedding $G \hookrightarrow \Ort_6(\mathbf F_2)^-$.
Since $G$ is simple, $G\subset [\Ort_6(\mathbf F_2)^-, \Ort_6(\mathbf F_2)^-]$.
It is known that $[\Ort_6(\mathbf F_2)^-, \Ort_6(\mathbf F_2)^-]\simeq \PSp_4(3)$
(see e.g.  \cite{atlas}). Moreover, 
$G$ is isomorphic to one of the following groups: $\PSp_4(3)$, $\Alt_6$ or $\Alt_5$.
On the other hand, $G$ faithfully acts on 
$H^0 (X, -K_X)\simeq \CC^4$. 
Then the only possibility is $G\simeq \Alt_5$. The defining equation $\psi(z)=0$ of $X\subset \PP^3$ 
is a cubic invariant on 
$H^0 (X, -K_X)$. Hence the representation on $H^0 (X, -K_X)$ is the standard irreducible
representation of $\Alt_5$. Then in a suitable basis we have $\psi=z_1^3+\cdots+z_4^3-(z_1+\dots+z_4)^3$.
The fixed point locus of an element $\upalpha\in G$ of order two is a union of a line and three 
isolated points. By the topological Lefschetz fixed point formula the action of $\upalpha$ on $\Lambda$ 
is diagonalizable as follows: $\upalpha=\diag(1,1,1,1,-1,-1)$. This implies that
the representation of $G$ on $\Lambda$ is the sum of the irreducible four-dimensional representation and 
the trivial one.
This contradicts the minimality assumption $\rk\Pic(X)^G=1$. 
\end{case*}

\begin{case*}{\bf Case $d=4$.}
Then $X = X_{2\cdot 2} = Q' \cap Q'' \subset \PP^4$ is an intersection of 
two quadrics. The group
$G$ acts on the pencil of quadrics $\langle Q', Q''\rangle \simeq \PP^1$
leaving invariant the subset of
five singular elements. It is easy to see that in this case 
the action on $\langle Q', Q''\rangle$ must be trivial. 
Hence, $G$ fixes vertices $P_1,\dots,P_5$ of five $G$-stable quadratic cones $Q_i\in\langle Q', Q''\rangle$.
Since these points $P_1,\dots,P_5$ generate $\PP^5$, $G$ must be abelian,
a contradiction.
\end{case*}

\begin{case*}{\bf Case $d=5$.}
A del Pezzo surface of degree $5$ is unique up to isomorphism.
Consider the (faithful) action of $G$ on the space $\Pic(X)\otimes \CC$ and on 
the orthogonal complement $K_X^\perp\subset \Pic(X)\otimes \CC$.
The intersection form induces a non-degenerate quadratic form on $K_X^\perp$.
Hence $G$ faithfully acts on a two-dimensional quadric in $\PP^3$.
Then the only possibility is $G\simeq \Alt_5$. 
One can see that $\Aut(X)$ isomorphic to the symmetric group $\Sym_5$ and so 
a del Pezzo surface of degree $5$ admits an action of $\Alt_5$.
We get the case \ref{theorem-2-dimensional-simple-5}.
\end{case*}

\begin{case*}{\bf Case $6\le d \le 8$.}
Then the action of $G$ on $\Pic(X)\simeq \ZZ^{10-d}$ 
must be trivial. This contradicts the minimality assumption $\rk\Pic(X)^G=1$. 
\end{case*}

\begin{case*}{\bf Case $d=9$.}
Then $X =\PP^2$. So $G \subset \PGL_3( \CC)$, and 
by the classification of finite subgroups in $\PGL_2( \CC)$
we get the case \ref{theorem-2-dimensional-simple-P2}. 
\end{case*}
Theorem \ref{theorem-2-dimensional-simple} is proved.
\end{proof}

\begin{theorem}[{\cite{Dolgachev-Iskovskikh}}, {\cite{Tsygankov2013}}]
\label{theorem-2-dimensional}
Let $G\subset \Cr_2(\CC)$ be a finite quasi-simple non-simple subgroup.
Then $G\simeq 2.\Alt_5$.
\end{theorem}
\begin{proof}
As in the proof of Theorem \ref{theorem-2-dimensional-simple} we may assume that 
$G$ faithfully acts on a smooth projective rational surface $X$ which is either 
a del Pezzo surface with $\rk \Pic(X)^G=1$ or an equivariant conic bundle with 
$\rk \Pic(X)^G=2$.  Let $\bar G:=G/\z(G)$. Then $\bar G$ is a simple group 
acting on a rational surface $X/\z(G)$. Hence $\bar G$ is embeddable to 
$\Cr_2(\CC)$ and by Theorem \ref{theorem-2-dimensional-simple} we have $\bar 
G\simeq \Alt_5$, $\Alt_6$, or $\PSL_2(7)$. Therefore, as in the proof of 
Proposition \ref{Proposition-list-intro} we have one of the following 
possibilities: $G\simeq 2.\Alt_5$, $\SL_2(7)$, or $n.\Alt_6$ for $n=2$, $3$, or 
$6$. If $X$ has an equivariant conic bundle structure $\pi: X\to B\simeq \PP^1$, 
then $G$ non-trivially acts  either on the base $B$ or on the general fiber. 
This is possible only if $G\simeq 2.\Alt_5$.

Assume that $X$ is a del Pezzo surface with $\rk \Pic(X)^G=1$. 
Let $Z\subset \z(G)$ be a cyclic subgroup of prime order $p$
and let $\pi: X\to Y:=X/Z$ be the quotient.
The surface $Y$ is rational and $\bar G:= G/Z$ faithfully acts on $Y$.

First, consider the case where $Z$ has only isolated fixed points. If $p=2$, then by the holomorphic
Lefschetz formula the number of fixed points equals $4$.
These points cannot be permuted by $G$, so they are fixed by $G$.
Similarly, in the case $p=3$ denote by $n_0$ (resp., $n_1$, $n_2$) 
the number of fixed points with action of type $\frac 13(1,-1)$
(resp., $\frac 13(1,1)$, $\frac 13(-1,-1)$).
Then again by the holomorphic Lefschetz formula
$n_1=n_2$, $n_0+n_1=3$. Hence there are at most three points of each type
and so, as above, all these points are fixed by $G$.
Since the groups $\SL_2(7)$ and $n.\Alt_6$ cannot act faithfully on the
tangent space to a fixed point, the only possibility is $G\simeq 2.\Alt_5$. 

Now consider the case where the fixed point locus $X^Z$ of $Z$
is one-dimensional.
Let $C$ be the union of all the curves in $X^Z$.
Note that $C$ is smooth because it is 
a part of the fixed point locus.
Clearly, $C$ is  $G$-invariant.
Then the class of $C$ must be proportional to $-K_X$
in $\Pic(X)$. Hence $C$ is ample and  connected.
Since $C$ is smooth, it is irreducible.
The group $G/Z$ non-trivially acts on $C$. Hence $C$ cannot be an elliptic curve.
Moreover, if $C$ is rational, then $G/Z\simeq \Alt_5$ and we are done.
Thus we can write $C\sim -aK_X$ with $a>1$.
If $-K_X$ is very ample, then $C$ is contained in 
a hyperplane section and $a=1$, a contradiction.
Thus it remains to consider only two possibilities: $K_X^2=1$ and $2$.
If $K_X^2=1$, then 
the anticanonical linear system has a unique base point, say $O$.  
Since, the representation of $G$ in the tangent space $T_{O,X}$ is faithful, 
the only possibility is $G\simeq 2.\Alt_5$.
Finally assume that $K_X^2=2$.
Then the anticanonical map is a double cover 
$\Phi_{|-K_X|}: X\to \PP^2$. The action of $Z$ of $\PP^2$ must be trivial
(otherwise $a=1$). Hence $p=2$, $Z$ is generated by the Geiser involution $\gamma$,
and $C$ is the ramification curve of $\Phi_{|-K_X|}$.
On the other hand, there exists the following homomorphism
\[
\lambda: \Aut(X) \hookrightarrow \GL(\Pic(X))=\GL_8(\ZZ)\overset{\det}\longrightarrow \{\pm 1\},
\]
where $\lambda(\gamma)=-1$.
Since our group $G$ is perfect, $\gamma\notin G$, a contradiction.
\end{proof}
\begin{remark}
In the same   manner one can describe the actions of $G=2.\Alt_5$
on rational surfaces, i.e. the embeddings $2.\Alt_5\hookrightarrow \Cr_2(\CC)$),
see \cite{Tsygankov2013} for details.
\end{remark}
\par\medskip\noindent
\textbf{Acknowledgements.} I would like to thank the referee for 
careful reading and numerous helpful comments and suggestions. 

\newcommand{\etalchar}[1]{$^{#1}$}
\def\cprime{$'$}

\end{document}